\documentclass[]{article}

\usepackage{turnstile}
\usepackage{amsmath}
\usepackage{amssymb}
\usepackage{amsthm}
\usepackage{tikz-cd}

\usepackage[utf8]{inputenc}

\newtheorem{cor}{Corollary}
\newtheorem{defin}{Definition}
\newtheorem{lem}{Lemma}
\newtheorem{thm}{Theorem}

\newtheorem{prob}{Problem}
\newtheorem{prop}{Proposition}
\newtheorem*{sub}{Sublemma}
\newtheorem{claim}{Claim}

\newcommand{\Lo}{\mathbb{L}^{\omega_1}}

\newcommand{\Ldd}{\mathbb{L}^{\omega_1}_{(D,d_0)}}

\newcommand{\ldim}{\operatorname{L-dim}}
\newcommand{\idim}{\operatorname{I-dim}}
\newcommand{\col}{\operatorname{Col}}

\title{On countably saturated linear orders and certain class of  countably saturated graphs}
\author{Ziemowit Kostana\footnote{Research of Z. Kostana was supported by the GAČR project 16-34860L and RVO: 67985840.}\\
University of Warsaw, \\The Czech Academy of Sciences, Institute of Mathematics\\
z.kostana@mimuw.edu.pl}
\begin{document}

\maketitle

\begin{abstract}

	 The idea of this paper is to explore the existence of canonical countably saturated models for different classes of structures.\\
	  It is well-known that, under CH, there exists a unique countably saturated linear order of cardinality $\mathfrak{c}$. We provide some examples of pairwise non-isomorphic countably saturated linear orders of cardinality $\mathfrak{c}$, under different set-theoretic assumptions. We give a new proof of the old theorem of Harzheim, that the class of countably saturated linear orders has a uniquely determined one-element basis. From our proof it follows that this minimal linear order is a Fra\"iss\'e limit of certain Fra\"iss\'e class. In particular, it is homogeneous with respect to countable subsets.\\
	Next we prove the existence and uniqueness of the uncountable version of the random graph. This graph is isomorphic to $(H(\omega_1),\in \cup \ni)$, where $H(\omega_1)$ is the set of hereditarily countable sets, and two sets are connected if one of them is an element of the other.\\
	In the last section, an example of a prime countably saturated Boolean algebra is presented.

\end{abstract}

{\bf Keywords:} linear order, countably saturated, homogeneous object, random graph \\
{\bf MSC classification:} 03C55, 06A05, 05C63

\section{Introduction}

It is a general theorem of model theory, that for any first-order theory which has infinite models, all countably saturated models of cardinality $\mathfrak{c}$ are isomorphic, as long as we assume CH. The definition of countable saturation and the proof of this theorem can be found in any textbook on model theory, for example \cite{hod}.
However, we will not bother with realization of types in general, and focus on some particular cases. These cases are:

\begin{itemize}
	\item (Hausdorff) Assume CH. Then all countably saturated linear orders (see next section for the definition) of cardinality $\mathfrak{c}$ are isomorphic.
	\item  Assume CH. There exists a unique up to isomorphism graph $G$ of cardinality $\mathfrak{c}$, with the property that for any countable, disjoint subsets $A, B \subseteq G$, there exists a vertex in $G\setminus (A \cup B)$ connected to every vertex in $A$ and none of the vertices in $B$.
	\item (Parovi\v{c}enko) Assume CH. Then all Boolean algebras of cardinality $\mathfrak{c}$ with the strong countable separation property are isomorphic. (for the definition see \cite{bool} or Section 3 of this article).
	
\end{itemize}

All of the above claims are proved using so-called back-and-forth argument, and all of them fail without CH. The fact that Parovi\v{c}enko Theorem is equivalent to CH was proved by van Douwen and van Mill in \cite{van}. The fact that each of the other two claims is equivalent to CH is much easier to see.  However, it turns out that some applications of back-and-forth argument can be carried out only assuming that $\mathfrak{c}$ is regular, or even in ZFC alone. For example, Avil\'es and Brech generalized Parovi\v{c}enko Theorem, by introducing a property stronger than the countable separation property, which guarantees the uniqueness of certain Boolean algebra with this property, when $\mathfrak{c}$ is regular \cite{ab}. A different generalization was obtained by Dow and Hart \cite{cohen}. Majority of structures obtained this way are examples of uncountable Fra\"iss\'e limits, which general theory is described in \cite{fraisse}. The aim of this paper is to study possible generalizations of the above results in the absence of CH.
 
\begin{thm}[Harzheim, \cite{harz}]
	There exists a unique up to isomorphism linear order of cardinality $\mathfrak{c}$, which is countably saturated, and embeds into any countably saturated linear order. 
\end{thm}

\begin{thm}
	There exists a unique up to isomorphism graph $G$ of cardinality $\mathfrak{c}$, such that
	\begin{itemize}
		\item  For any disjoint sets $A,B \subseteq G$, $|A|\le \omega$, $|B|<\mathfrak{c}$, there exist a vertex $g\in G\setminus(A\cup B)$, such that $g$ is connected to every element of $A$ and none of the elements in $B$,
		\item $G$ embeds into any graph with the above property.
	\end{itemize}
\end{thm}

\begin{thm}
	There exists a unique up to isomorphism directed graph $\overleftarrow{G}$ of cardinality $\mathfrak{c}$, such that
	\begin{itemize}
		\item  For any countable set $A \subseteq \overleftarrow{G}$ there exist $\mathfrak{c}$-many $g\in \overleftarrow{G}$, such that 
		$A=\{a\in \overleftarrow{G}|\, (g,a)\in E(\overleftarrow{G})\}$,
		\item For any $g \in G$, the set $A=\{a\in \overleftarrow{G}|\, (g,a)\in E(\overleftarrow{G})\}$ is countable,
		\item $\overleftarrow{G}$ embeds into any graph with the above properties.
	\end{itemize}
\end{thm}

 The author would like to thank Wies\l{}aw Kubi\'s for inspiring conversations and asking many good questions which laid foundation for this work, and the anonymous referee for pointing out numerous mistakes in the first version of the paper. Also some credit goes to Arturo Martinez Celis, who came up with the idea of representation of the graph from Theorem 2.

\section{Countably saturated linear orders}

In this section we give the definition and introduce some examples of countably saturated linear orders. Next, we introduce the notion of linear dimension,
which we use to characterize those countably saturated linear orders, which are embeddable into any countably saturated linear order. Finally,
we prove that the linear order with this property is unique up to isomorphism.

We use the distinction \textit{non-decreasing/increasing} rather than \textit{increasing/strictly increasing}. For linear orders $(K,\le_K)$, $(L,\le_L)$, we define $K+L$ as $\{0\}\times K \cup \{1\}\times L$, with the lexicographic order. For $A,B \subseteq L$, by $A<B$, we denote that $\forall_{x\in A}\forall_{y\in B} x<y$. Later, whenever a product of linear orders is considered, we equip it with lexicographic order without mentioning it. By symbol $K \hookrightarrow L$ we denote that there is an increasing function from $K$ to $L$. By $L^*$ we denote $L$ with the reversed order, i.e. $(L,\le)^*=(L,\ge)$.\\

A linear order is \emph{compact}, if it is compact in the order topology, i.e. every set has both supremum and infimum, in particular both endpoints exist. We often call such orders \emph{compact lines}.\\

A linear order is a linearly ordered continuum, if it is compact and connected in the order topology, i.e. it is compact and dense as a linear order. \\

A linear order $L$ has character at most $\kappa$, if for each $l \in L$, there exists subsets $A,B \subset L$ of cardinality at most $\kappa$, such that $A<B$,\\ and $\{x\in L|\; A<\{x\}<B \}=\{l\}$. As to be expected, $L$ has character $\kappa$, if $\kappa$ is the least cardinal, such that $L$ has character at most $\kappa$. We will make use of the following observation a few times. The proof is left for the Reader.

\begin{prop}
	If a linear order $L$ has uncountable character, then $\omega_1 \hookrightarrow L$ or $\omega_1^* \hookrightarrow L$.
\end{prop}

\subsection{Basic examples}

\begin{defin}
	We say that a linear order $(L,\le)$ is countably saturated, if for any countable linear orders $a$, $b$, and increasing functions $i:a\rightarrow b$, $f:a\rightarrow L$, there exists $\tilde{f}:b\rightarrow L$, such that
	$\tilde{f}\circ i=f$.
	
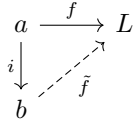
	\begin{tikzcd}
	a \arrow{r}{f} \arrow{d}[swap]{i} & L\\
	b \arrow[dashed]{ur}[swap]{\tilde{f}}
	\end{tikzcd}
	\\
	Moreover, $L$ is prime countably saturated, if any countably saturated linear order contains an isomorphic copy of $L$.
\end{defin}

The following simple Lemma provides more operational characterization, and justifies the name "saturated" - it is just saturation in the usual, model-theoretic sense.
\begin{lem} A linear order is countably saturated if and only if
	\begin{itemize}
		\item it is dense, without the least and the greatest element,
		\item no countable increasing sequence has supremum,
		\item no countable decreasing sequence has infimum,
		\item there are no $(\omega,\omega)$-gaps: for any two sequences $\{x_n\}_{n<\omega}$, $\{y_n\}_{n<\omega}$ such that $\forall{n<\omega}\;{x_n<x_{n+1}<y_{n+1}<y_n}$, exists $z$ s.t. $\forall{n<\omega}\;{x_n<z<y_n}.$
			
	\end{itemize}
\end{lem}

\begin{proof}
	Suppose, that $L$ satisfies all of the above, $a$ and $b$ are countable, and $f:a \rightarrow L$ is an embedding, which we want to extend to $\tilde{f}:b\rightarrow L$. Notice, that it is sufficient to show this in case, when $b$ is a one-point extension of $a$, since then we can do so for any countable linear order by induction. But conditions from the Lemma assure exactly, that any one-point extension of $a$ can be realized in $L$.
	\\ In the other direction, verification is straightforward.
\end{proof}

Linear orders with this property were investigated already by Hausdorff \cite{Hausdorff}, under the name $\eta_1$-orders, and later by Harzheim \cite{harz}. A brief survey with substantial bibliography can be found in Chapter 9 of the book by Rosenstein \cite{rosen}.

Since each one-point extension of a countable subset of countably saturated linear order is realised, each linear order of cardinality at most $\omega_1$ can be embedded into any countably saturated linear order. Moreover, the following stregthening of this observation is attributed to Hausdorff in \cite{Urysohn}, and \cite{rosen} p.164.

\begin{thm} A countably saturated linear order contains a copy of any linear order, which doesn't have a copy of $\omega_1$ or $\omega_1^*$.
\end{thm}

In particular, each countably saturated linear order contains an isomorphic copy of the real line order type, therefore has cardinality at least $\mathfrak{c}$. 

Let us now define a linear order, which will turn out to be prime countably saturated.

$$\mathbb{L}^{\omega_1}=\{x\in[-1,1]^{\omega_1}|\,|\{\alpha<\omega_1:\, x(\alpha)\neq 0\}|\le \omega\}.$$ If $D$ is a compact linear order and $d_0\in D$ is neither the least, nor the greatest element of $D$, then we define
$$\mathbb{L}_{(D,d_0)}^{\omega_1}=\{x\in D^{\omega_1}|\,|\{\alpha<\omega_1:\, x(\alpha)\neq d_0\}|\le \omega\}.$$

The case $D=\{-1,0,1\}$ and $d_0=0$ is the classical construction by Hausdorff, and is described in \cite{harz}. Theorem 5 is therefore a slightly more general version of Theorem 3.22 therein, and Theorem 7 - a version of Theorem 3.13.

\begin{thm}
	$\mathbb{L}^{\omega_1}$ and $\mathbb{L}_{(D,d_0)}^{\omega_1}$ are countably saturated.
\end{thm}

This follows in fact almost immediately from a theorem of Nov\'ak.

\begin{thm}[Nov\'ak, \cite{novak}] Let $L$ be a compact line. Then for any ordinal $\alpha$, $L^\alpha$ is compact. If $L$ was moreover a continuum, then so is $L^\alpha$.
\end{thm}

\begin{proof}
	We prove the theorem by induction on $\alpha$. Case $\alpha=1$ is trivial, and for the succesor step it is sufficient to show that the product of two compact linear orders is compact, and the product of two continua is a continuum. \\
	Let $(K,\le_K)$ and $(L,\le_L)$ be compact lines. Surely $(\min{K},\min{L})$ is the minimal element of $K\times L$, and
	$(\max{K},\max{L})$ is the greatest element. Let $S\subseteq K \times L$ be a non-empty set. We aim to find the least upper bound for $S$. If projection on the first coordinate $\operatorname{proj}_K{S}$ has no maximal element, then $(\sup{\operatorname{proj}_K{S}},\min{L})$ is the least upper bound for $S$. If $\operatorname{proj}_K{S}$ has a maximal element, say $s_K$, then  let $s_L$ be the least upper bound of $\{x \in L|\; (s_K,x) \in S\}$. Then $(s_K,s_L)$ is the least upper bound for $S$ in $K\times L$. In a similar way we prove that $S$ has the greatest lower bound.
	\\
	Assume now that $\alpha$ is a limit ordinal, and for every $\beta<\alpha$, each lexicographic power $L^\beta$ is compact. Then
	the space $\displaystyle{\prod_{\beta<\alpha}\; L^{\beta}}$ with the product topology (not the \emph{lexicographic} power) is a compact Hausdorff space, and so the set
	$$P=\{\overline{x}\in \prod_{\beta<\alpha}L^\beta|\; \forall{\beta_1<\beta_2<\alpha}\thickspace
	\overline{x}(\beta_2)\restriction \beta_1 = \overline{x}(\beta_1) \}$$
	is also compact, as a closed subset of a compact space. It is then sufficient to show that the lexicographic power $L^\alpha$ is a continuous image of $P$. The obvious function witnessing that is given by $\pi(\overline{x})(\beta)=\overline{x}(\beta+1)(\beta)$. To see that $\pi$ is continuous, it is enough to check that $\pi^{-1}[\{x\in L^\alpha|\; x>a \}]$ is open for every $a \in L^\alpha$. But since inequality $x>a$ must be true already between some initial segments
	of $x$ and $a$, it is easy to check that
	$$\pi^{-1}[\{x\in L^\alpha|\; x>a \}]=\{\overline{x} \in \prod_{\beta<\alpha}L^\beta|\;
	\exists{\gamma<\alpha}\; \overline{x}(\gamma)>a\restriction \gamma  \} \cap P.$$
	The latter set is clearly open in $P$. The claim about continua is straightforward, given that a compact line is a continuum if and only if it is dense in itself.
\end{proof}

\begin{proof}[Proof of Theorem 5]

Notice, that $\Lo=\displaystyle{\bigcup_{\alpha<\omega_1}{[-1,1]^\alpha}}$. We check conditions from Lemma 1. Density is clear. For verification of the last condition, take two sequences 
$\ldots <x_n<x_{n+1}<\ldots<y_{n+1}<y_n<\ldots$. There is some level $\alpha$, such that all $x_n$ and $y_n$ belong to $[-1,1]^\alpha$. Since it is compact, we can take supremum of $\{x_n\}_{n<\omega}$ in this set, and it will clearly separate each $x_n$ from each $y_n$.\\
The only problematic case is when we want to separate an increasing sequence $\{x_n\}_{n<\omega}$ from $\tilde{x}=\sup_{n<\omega}{x_n}$, where supremum is taken in some $[-1,1]^\alpha$, big enough to contain $\tilde{x}$ and each $x_n$. But then it is clear, that $\tilde{x}^\frown{-1} \in [-1,1]^{\alpha+1}$ is a good separating element. Proof for $\mathbb{L}_{(D,d_0)}^{\omega_1}$ goes exactly the same way.
\end{proof}

\begin{thm}
	$\Lo$ is prime countably saturated. Likewise, if $D$ is a separable compact line, and $d_0 \in D$ is neither the least, nor the greatest element, then $\Ldd$ is prime countably saturated.
\end{thm}

  $\Lo=\{x\in[-1,1]^{\omega_1}|\,|\{\alpha<\omega_1:\, x(\alpha)\neq 0\}|\le \omega\}$ is (isomorphic to) an increasing sum $\bigcup_{\alpha<\omega_1}{[-1,1]^{\alpha}}$. If $X$ is any countably saturated linear order, we build an embedding $\Lo \rightarrow X$ by induction on $\alpha$, using the Lemma below.

\begin{lem}
	Let $(X,\le_X)$ be any countably saturated linear order, $(L,\le_L)$ a linear order of countable character, and $(E,\le_E)$ be separable linear order with distinguished element $e_0$. Then each embedding $i:\, L\simeq L\times\{e_0\}\rightarrow X$ extends to $\tilde{i}:L\times E\rightarrow X$.
\end{lem}
\begin{proof}
	Det $E_0\subseteq E$ be a countable dense subset. Fix $l\in L$ and sequences $\alpha_n$ increasing, $\beta_n$ decreasing, witnessing countable character of $l$. Then, there exists an extension:\\
	\begin{tikzcd}
		(\{l\}\cup\{\alpha_n,\beta_n|\, n<\omega\},\le_L) \arrow{r}{i} \arrow{d}{j} & X\\
		(\{\alpha_n\}_{n<\omega},\le_L)+E_0+(\{\beta_n\}_{n<\omega},\le_L) \arrow[dashed]{ur}[swap]{i_l}
	\end{tikzcd}\\
where $j(l)=e_0$, $j(\alpha_n)=\alpha_n$. and $j(\beta_n)=\beta_n$ for every $n$. \\
Next, we extend $$i_l:(\{\alpha_n\}_{n<\omega},\le_L)+E_0+(\{\beta_n\}_{n<\omega},\le_L) \rightarrow X,$$ to $$\tilde{i}_l: (\{\alpha_n\}_{n<\omega},\le_L)+E+(\{\beta_n\}_{n<\omega},\le_L).$$

This, in turn, can be done, because $E_0$ is countable and dense in $E$. Finally, we set $\tilde{i}(l,e)=i_l(e)$.
\end{proof}

Now we only need to know, that the sets $[-1,1]^\alpha$ have countable character. 

\begin{prop}
	If $\alpha$ is a countable ordinal, then $[-1,1]^\alpha$ doesn't contain any copy of $\omega_1$ or $\omega_1^*$. In particular, it has countable character.
\end{prop}

\begin{proof}
	It is sufficient to prove that $2^\gamma$ doesn't contain copy of $\omega_1$, for any countable ordinal $\gamma$. Assume otherwise, and let $\gamma<\omega_1$ be minimal, such that $\omega_1 \hookrightarrow 2^\gamma$. \\
	Suppose that $\gamma$ is limit. It is easy to check that the set
	$$2^{<\gamma}=\{x\in2^\gamma|\, \exists_{\beta<\gamma}\forall{\beta<\delta<\gamma}\; x(\delta)=0 \}=\bigcup_{\delta<\gamma}{2^\delta}$$
	is dense in $2^\gamma$, and so if $2^\gamma$ has an uncountable well-ordered sequence, so does $2^{<\gamma}$. But then for some $\beta<\gamma$, $2^\beta$ has uncountably many elements of that sequence, which itself constitute a copy of $\omega_1$. This contradicts the minimality of $\gamma$.\\
	Suppose now that $\gamma$ is successor. If $2^\gamma$ has an uncountable well-ordered sequence, then either $\omega_1$ of its elements has
	$0$ on the last coordinate, or $\omega_1$ of its elements has $1$ on the last coordinate. Either way, their restriction to the first $\gamma-1$ coordinates constitute an uncountable well-ordered sequence in $2^{\gamma-1}$. This again contradicts the minimality of $\gamma$.
\end{proof}

We address the question of uniqueness of a prime countably saturated order. This question is completely settled under CH, and it was proved by Hausdorff. It is in fact a particular case of a much more general phenomenon. Namely, given any first order theory, all $\kappa$-saturated models of cardinality $\kappa$ are pairwise isomorphic. For the proof of this well-known fact, we refer to for example \cite{hod}, or to \cite{fraisse} for a more general approach, using the language of category theory.

\begin{thm} Assume CH. Then every two countably saturated linear orders of cardinality $\mathfrak{c}$ are isomorphic.
\end{thm}
	Let us note, that this is not true without CH. Namely, if $2^\omega \ge \omega_2$, then both 
	$$\{x\in[-1,1]^{\omega_1}|\,|\{\alpha<\omega_1:\, x(\alpha)\neq 0\}|\le \omega\},$$ and
	$$\{x\in[-1,1]^{\omega_2}|\,|\{\alpha<\omega_2:\, x(\alpha)\neq 0\}|\le \omega\},$$
	are countably saturated linear orders of cardinality continuum, but are not isomorphic, since the latter contains an isomorphic copy of the ordinal $\omega_2$. This was noted already in \cite{Gillman}. In many cases we can provide somewhat better example. For the forcing terminology we refer the reader to \cite{kunen}.
	
\begin{thm} It is relatively consistent with ZFC that there exist two non-isomorphic, countably saturated linear orders of cardinality $\mathfrak{c}$, none of which contains copy of neither $\omega_2$ nor $\omega_2^*$.
\end{thm}

\begin{proof}

We start with some countable transitive model $\mathbb{V} \models ZFC+CH$.
The idea is that we will find two linear orders satisfying conditions of the Theorem, but only one will contain a copy of some linear order of cardinality $\omega_2$ from $\mathbb{V}$.
Let $\mathbb{P} \in \mathbb{V}$ be any c.c.c. forcing notion which forces $\mathfrak{c}>\omega_1$ (for example a finite support iteration of the Cohen forcing, of length $\omega_2$). Denote by $\mathbb{V}^\mathbb{P}$ the corresponding generic extension. 

\begin{lem} 
		In $\mathbb{V}^\mathbb{P}$, $\Lo$ doesn't contain copy of any linear order of cardinality $\omega_2$, which is in $\mathbb{V}$.
\end{lem}

\begin{proof}
Suppose that $(\omega_2,\preccurlyeq) \in \mathbb{V}$ is a linear order, and that
$$\mathbb{P} \Vdash \text{"$\dot{f}:(\omega_2,\preccurlyeq)\hookrightarrow \Lo$ is increasing"}.$$
For any $\alpha \in \omega_2$, in the generic extension ${f}(\alpha)$ is a sequence of reals
of length $\omega_1$, which is constantly equal zero from some point. From the Maximum Principle, there exists a name $\dot{b_\alpha}$ for which
$\mathbb{P} \Vdash \forall{\omega_1>\delta>\dot{b_\alpha}}\; \dot{f}(\alpha)(\delta)=0.$ Since $\mathbb{P}$ is c.c.c. there are at most countably many possible values of the ordinal $\dot{b_\alpha}$, and so we can define $c_\alpha=\sup{\{ \beta<\omega_1|\; \exists{p\in \mathbb{P}}\; p \Vdash \dot{b_\alpha}=\beta
 \}}$. \\
There is some $\beta < \omega_1$, for which the set $\{ \alpha < \omega_2|\; c_\alpha=\beta \}$ has cardinality $\omega_2$. Denote this set by $S$. Note, that in this case
$$\mathbb{P} \Vdash \dot{f}[S]\hookrightarrow [-1,1]^\beta.$$
The set $S$ was defined in $\mathbb{V}$, so $S \in \mathbb{V}$, and given that $|S|=\omega_2$, $S$ must contain an uncountable well-ordered or reversed well-orered sequence (this is a standard application of the Erd\"os-Rado Theorem in $\mathbb{V}$). But $[-1,1]^\beta$ doesn't contain uncountable (reversed) well-ordered sequences, by Proposition 2. This is a contradiction.

\end{proof}

The Theorem is now proved as follows. We work in $\mathbb{V}^\mathbb{P}$.
As one example take $\Lo$. For the other, let $R_0 \subseteq (2^{\omega_1})^\mathbb{V}$ be of cardinality $\omega_2$, and $R_0 \in \mathbb{V}$. Define inductively a sequence $(R_\alpha)_{\alpha\le\omega_1}$, such that for each $\alpha<\omega_1$, $R_{\alpha+1}$ patches $R_\alpha$ (see Definition 3 in the next section), and $R_\alpha=\displaystyle{\bigcup_{\beta<\alpha}{R_\beta}}$ for limit $\alpha$. Moreover, we assure that $|R_\alpha|\le \mathfrak{c}$ for each $\alpha$. It is clear, that $R_{\omega_1}$ is countably saturated linear order of cardinality $\mathfrak{c}$.
\end{proof}

\subsection{Linear dimension}

We will use the notion of dimension, introduced by Nov\'ak in \cite{novak}, for further classification of linear orders. Some time ago, it was also investigated by Giarlotta, under the name \textit{representability number} \cite{giarlotta}. Let $2$ denote the 2-element linear order.

\begin{defin}[Nov\'ak, \cite{novak}] Let $L$ and $X$ be linear orders, and $|L|\ge 2$. The dimension of $X$ with respect to $L$ is defined as
	$$\ldim{X}=\min\{\alpha \in ON|\; X\hookrightarrow L^\alpha\}.$$
\end{defin}

Let us write down some easy observations.

\begin{prop} For any linear orders $L,L_1,L_2,X$, the following holds.
	\begin{itemize}
		\item If $X_1 \hookrightarrow X_2$, then $\ldim{X_1}\le \ldim{X_2}$.
		\item If $L_1\hookrightarrow L_2$, then 
		$\operatorname{L_1-dim}{X}\ge\operatorname{L_2-dim}{X}$.
		\item If $L_1\hookrightarrow L_2$ and $L_2\hookrightarrow L_1$, then for every $X$, $\operatorname{L_1-dim}{X}=\operatorname{L_2-dim}{X}$.
	\end{itemize}
\end{prop}

In particular, the notions of $\operatorname{2^\omega-dim}{X}, \operatorname{I-dim}{X}$, and $\operatorname{\mathbb{R}-dim}{X}$ coincide. We will denote them $\idim{X}$.\\

\begin{prop}
	If $A$ and $B$ are subsets of some linear order, $\idim{A}<\omega_1$, and $\idim{B}<\omega_1$, then $\idim{(A\cup B)}<\omega_1$.
\end{prop}

\begin{proof}
	Because $A\cup B=(A\setminus B) \, \cup (A\cap B) \, \cup (B \setminus A)$, and the sum is clearly disjoint, we can restrict ourselves to the case when $A$ and $B$ are disjoint. For $b_0,b_1 \in B$ we set $b_0 \sim b_1$ iff $[b_0,b_1]\cap A =\emptyset$ ($[b_0,b_1]$ actually means $[\min\{b_0,b_1\},\max\{b_0,b_1\}]$, but this is a harmless abuse of notation, and we won't bother with it anymore).
	This is a convex equivalence relation on $B$.  Let $i_A:A\rightarrow I^\alpha$, and $i_B:B\rightarrow I^\beta$ be embeddings. We aim to extend $i_A$ to $\tilde{i}:A\cup B\rightarrow I^{\alpha+1+\beta}$. It is sufficient to define $\tilde{i}$ on each equivalence class separately. Let $S\subseteq B$ be a selector of $\sim$. For $s\in S$ we choose $i'(s) \in I^{\alpha+1}$, so that 
	$$\{i_A(a)|\, a<s\}<\{i'(s)\}<\{i_A(a)|\, s<a\}.$$ 
	Keeping in mind that $I^\alpha$ is a linearly ordered continuum of countable character, it requires only standard verifiation, that this can be done, and it is left to the Reader. Let $i'(a)=i_A(a)$ for $a\in A$. Then, having $i':A\cup S \rightarrow I^{\alpha+1}$, we extend it to $\tilde{i}:A\cup B \rightarrow I^{\alpha+1+\beta}$, via the formula
	$$\tilde{i}(b)=i'(s_b)^\frown i_B(b),$$
	for $s_b\sim b$, and $s_b\in S$.
\end{proof}

\begin{thm}[Novotn\'y \cite{novotny}, Nov\'ak \cite{novak}] Let $L$ be a linearly ordered continuum. Then for any ordinal $\alpha$, $\ldim{L^\alpha}=\alpha$.
\end{thm}

Once we are familiar with the theorem of Nov\'ak, this result is an immediate consequence of a Lemma due to Novotn\'y, which has quite pleasant proof. Since it was originally published in Czech, we quote it here.
	
	\begin{lem}[Novotn\'y, \cite{novotny}]
		Let $(X,\le)$ be a linearly ordered continuum, and $\mathcal{A}$ be a disjoint family of closed, not-single-point intervals in $X$. The relation $\le$ induces a linear order $\tilde{\le}$ on $\mathcal{A}$, via the formula
		$I\tilde{\le}J$ iff $\max{I}<\min{J}$. Then $(X,\le)$ does not embed into $(\mathcal{A},\tilde{\le})$.
	\end{lem}	
	\begin{proof}
	Suppose otherwise, that $i:(X,\le)\rightarrow (\mathcal{A},\tilde{\le})$ is an embedding. 
	Let $A=\{x\in X|\;  x<\min{i(x)} \}$. $A$ is nonempty, because $X$ has the least element, so let $a=\sup{A}$. We have two cases.
	\begin{itemize}
		\item[a)] Assume $a\in A$. Then $a<\min{i(a)}$, but for any $a'\in i(a)$ this is the case as well. Since $a'\in i(a)$ is greater than $a$, this contradicts the definition of $a$.
		\item[b)] Assume $a\notin A$. Then $a\ge \min{i(a)}$. If $a<\max{i(a)}$, then let $a' \in (a,\max{i(a)})$. Then $a'<\max{i(a)}<\min{i(a')}$, so $a'\in A$, and this is a contradiction. The only way out is $a\ge\max{i(a)}$. But notice, that in this case for any $a' \in (\min{i(a)},a)$, we have $a' \notin A$. Therefore $\sup{A}\le \min{i(a)}<a$, which is again a contradiction.
	\end{itemize}
\end{proof}

\begin{cor}
	If $\alpha$ is an ordinal with the property, that $\omega\cdot \alpha=\alpha$, \\then $\idim{2^\alpha}=\alpha$.
\end{cor}

\begin{proof}
	Clearly $2^\omega \hookrightarrow I$, so 
	$$2^\alpha \simeq 2^{\omega\cdot\alpha} \simeq (2^\omega)^\alpha \hookrightarrow I^\alpha$$
	In the other direction,
	$$I^\alpha\hookrightarrow (2^\omega)^\alpha\simeq 2^{\omega\cdot\alpha}
	\simeq 2^\alpha.$$
\end{proof}

\begin{thm} Assume $\mathfrak{c}=2^{\omega_1}$. Let $X=I^{\omega_1}$. Then $\mathbb{L}^{\omega_1}_{(X,0)}$ is a countably saturated linear order of cardinality $\mathfrak{c}$, character $\omega_1$, and $\idim=\omega_1^2$. In particular 
$\mathbb{L}^{\omega_1}_{(X,0)}$ is not isomorphic to $\Lo$.
\end{thm}

\begin{proof}
	For every ordinal $\alpha<\omega_1$, $\mathbb{L}^{\omega_1}_{(X,0)}$ contains a copy of $X^\alpha$. $X^{\alpha}$ is a linearly ordered continuum, so $\idim{X^\alpha}= \idim{I^{\omega_1\cdot\alpha}}=\omega_1\cdot\alpha$. Since $\alpha$ was arbitrary, this shows that $\idim{\mathbb{L}^{\omega_1}_{(X,0)}} \ge \omega_1^2$.
	\\ In the other direction, $\mathbb{L}^{\omega_1}_{(X,0)}\hookrightarrow X^{\omega_1} \hookrightarrow I^{\omega_1\cdot\omega_1}$, so $\idim{\mathbb{L}^{\omega_1}_{(X,0)}}=\omega_1^2$.
\end{proof}

\subsection{Uniqueness of the prime countably saturated linear order}

The assumption of CH is not needed if we take into account only linear orders which are in some sense minimal.

\begin{thm}[Harzheim, \cite{harz}]
	All prime countably saturated linear orders are isomorphic.
\end{thm}

The order type of the unique prime countably saturated order is in fact the order type $h_1$ in the language of \cite{harz}. Theorem 12 follows from Theorems 3.13 and 3.17 there. We give a proof in a slightly different language.

\begin{defin}
	For two linear orders $A\subseteq B$, we will say, that $B$ \textit{patches} $A$, if for any two countable sets $a_0, a_1 \subseteq A$, if $a_0<a_1$, there exists $b\in B$, such that $a_0<\{b\}<a_1$.
\end{defin}

One can easily verify, that this comes down to patching gaps of four different types: $(1,1),\;(\omega,1),\;(1,\omega),$ and $(\omega,\omega)$, which correspond to situations, where respectively
\begin{itemize}
	\item $a_0$ has the greatest element, and $a_1$ has the least element,
	\item $a_0$ doesn't have the greatest element, but $a_1$ has the least element,
	\item $a_0$ has the greatest element, but $a_1$ doesn't have the least element,
	\item neither $a_0$ has the greatest element, nor $a_1$ has the least element.
\end{itemize}

\begin{lem}
	If $\idim{A} < \omega_1$, then there exists $B\supseteq A$ patching $A$, such that $\idim{B}<\omega_1$.
\end{lem}
\begin{proof}
	Let $\alpha=\idim{A}$. Then $A \hookrightarrow I^{\alpha}$, so we can assume that $A \subseteq I^{\alpha}$. This order clearly patches gaps of the form $(1,1)$ and $(\omega,\omega)$. To take care of $(1, \omega)$- and $(\omega,1)$-gaps, we replace every point of $I^{\alpha}$ with a unit interval.
	$$A\subseteq I^{\alpha}\simeq I^{\alpha}\times\{0\}\subseteq I^{\alpha+1}.$$
\end{proof}

\begin{defin}
	If $L$ is a countably saturated linear order, we define a filtration of $L$, as a sequence of subsets of $L$, $\{L_\alpha\}_{\alpha<\omega_1}$, with the following properties:
	\begin{itemize}
		\item $\{L_\alpha\}_{\alpha<\omega_1}$ is increasing with respect to inclusion,
		\item for each $\alpha<\omega_1$, $L_{\alpha+1}$ patches $L_\alpha$,
		\item $\displaystyle{\bigcup_{\alpha<\omega_1}{L_\alpha}=L}$,
		\item for each $\alpha<\omega_1$, $\idim{L_\alpha}<\omega_1$.
	\end{itemize}
\end{defin}

The next Lemma is the key tool in the proof of Theorem 12. It should be noted, that its first part as actually a direct consequence of Theorem 3.6 p. 81 \cite{harz}.

\begin{lem}[On Bounded Injectivity]
	Assume, that $A$ and $B$ are linear orders with
	$\idim{A}, \idim{B}<\omega_1$, and $L$ is countably saturated. Let $f:A\hookrightarrow L$, and $i:A\hookrightarrow B$ be increasing functions. Then 
	\begin{enumerate}
		\item[a)] There exists an increasing mapping $\tilde{f}:A \hookrightarrow L$, such that $\widetilde{f}\circ i=f$.	
	
	\begin{tikzcd}
	A \arrow{r}{f} \arrow{d}[swap]{i} & L\\
	B \arrow[dashed]{ur}[swap]{\tilde{f}}
	\end{tikzcd}
	
	 	\item[b)] If $L$ has a filtration $\{L_\alpha\}_{\alpha<\omega_1}$, and there is an index $\alpha < \omega_1$, for which $f[A] \subseteq L_\alpha$, then we can choose $\tilde{f}$ bounded, i.e. $\tilde{f}[B]\subseteq L_\beta$, for some 
		$\beta<\omega_1$.
	\end{enumerate}
\end{lem}

\begin{proof}
	
	Without loss of generality we may assume that $A\subseteq B$, and $i$ is the identity mapping. For $b_0,b_1 \in B \setminus A$, we set $b_0 \sim b_1$ iff $[b_0,b_1]\cap A=\emptyset$. This clearly defines a convex equivalence relation on $B\setminus A$, and
	it is sufficient to define $\tilde{f}$ on each equivalence class separately. 
	\item[a)] For any $b\in B \setminus A$,
	$$\{a\in A|\,a<b\}<[b]_{\sim}<\{a\in A|\,b<a\}.$$
	Since any linear order of countable $\idim$ has character $\omega$, both these sets have countable cofinality and coinitiality respectively. Therefore, this is also the case for  	
	$$\{f(a)|\,a<b, \, a \in A\}<\{f(a)|\,b<a,\, a\in A\},$$
	so $L$ will contain a point between them. Since it will contain a point, it will contain an interval, and this interval will be countably saturated itself. Since it will be countably saturated, it will contain a copy of $\Lo$, and in turn a copy of any order of countable $I$-dimension, in particular $[b]_{\sim}$. We set $\tilde{f}$ to isomorphically map $[b]_{\sim}$ onto that copy.
	\par
	The "bounded" variant will require more care.
	\begin{sub}
	Let $\{L_\alpha\}_{\alpha<\omega_1}$ be a filtration. Then for every $\alpha<\omega_1$, $\omega \le \gamma <\omega_1$, and for all $a_0 < a_1 \in L_\alpha$, $(a_0,a_1)\cap L_{\alpha+\gamma+1}$ contains a copy of $2^\gamma$. 	
	\end{sub}
	\begin{proof}We proceed by induction on $\gamma$.
		\par
			\item[]\textit{$\gamma=\omega$.}
			$L_{\alpha+\omega} \cap (a_0,a_1)$ contains copy of rationals, so $L_{\alpha+\omega+1} \cap (a_0,a_1)$ contains a copy of reals, and in particular $2^\omega$.\\
			\item[]\textit{successor step.}
			Assume that $i:2^\gamma \hookrightarrow (a_0,a_1)\cap L_{\alpha+\gamma+1}$.For every $x\in 2^\gamma$, there exists $v_x \in L_{\alpha+\gamma+2}$, such that
			$$i(x)<v_x<i[\{y\in 2^\gamma|\,y>x\}],$$
			because the set on the right has countable coinitiality. We set $\tilde{i}(x^\frown 0)=i(x)$, and $\tilde{i}(x^\frown 1)=v_x$.\\
			\item[]\textit{limit step.}
			Fix an increasing sequence $\{\gamma_n\}_{n<\omega}$, cofinal on $\gamma$. Let 
			$$2^{<\gamma}=\{x\in 2^\gamma|\,\exists{\beta<\gamma}\forall{\beta<\delta<\gamma}\; x(\delta)=0\}=\bigcup_{n<\omega}{2^{\gamma_n}}.$$
			By induction we define embeddings $j_n:2^{\gamma_n}\hookrightarrow (a_0,a_1)\cap L_{\alpha+\gamma_n+1}$, such that $j_n\subseteq j_{n+1}$. \\
			$j_0:2^{\gamma_0}\rightarrow (a_0,a_1)\cap L_{\alpha+\gamma_0+1}$ exists by induction hypothesis. Assume, that we have $j_n$ defined, and look at $x\in 2^{\gamma_n}$. Since $L_{\alpha+\gamma_n+2}$ patches $L_{\alpha+\gamma_n+1}$, there exist $v_x$, such that
			$$j_n(x)<v_x<j_n[\{y\in 2^{\gamma_n}|\,y>x\}],$$
			and $(j_n(x),v_x)\subseteq (a_0,a_1)$ (this is automatic, unless $\{y\in 2^{\gamma_n}|\,y>x\}=\emptyset$).
			By the induction hypothesis we can find $l:2^{\gamma_{n+1}\setminus \gamma_n}\hookrightarrow (j_n(x),v_x)\cap L_{\alpha+\gamma_{n+1}+1}$. We define, for $w\neq 0$,
			$j_{n+1}(x^\frown w)=l(w)$. Notice, that $x^\frown w \in 2^{\gamma_n}\times 2^{\gamma_{n+1}\setminus\gamma_n}=2^{\gamma_{n+1}}$.\\
			Finally, we see that $\displaystyle{\bigcup_{n<\omega}{
			j_n:2^{<\gamma}\hookrightarrow L_{\alpha+\gamma}}}$,
			and so $2^\gamma\hookrightarrow L_{\alpha+\gamma+1}$.
	\end{proof}
	\item[b)]Assume that $\{L_\alpha\}_{\alpha<\omega_1}$ is a filtration of $L$, and $f[A]\subseteq L_\beta$. Then there exists $u,v \in L_{\beta+2}$, such that 
	$$\{f(a)|\,a<b\}<u<v<\{f(a)|\,b<a\}.$$ By the Sublemma, 
	$(u,v)\cap L_{\beta+\omega\cdot\gamma+1}$ contains a copy of $2^{\omega\cdot\gamma}$, so also $I^\gamma$, for any given $\gamma$. If $\gamma=\max\{\idim{B},\omega\}$, then in particular $(u,v)\cap L_{\beta+2+\omega\cdot\gamma+1}$ contains a copy of $[b]_{\sim}$. We define $\tilde{f}$ the same way as before, but making sure, that $\tilde{f}[B]\subseteq L_{\beta+\omega\cdot\gamma+1}$.
\end{proof}

\begin{prop}
	A countably saturated linear order is prime if and only if it has a filtration.
\end{prop}

\begin{proof}
	Assume, that $L$ is prime countably saturated. Then without loss of generality, we can assume, that $L \subseteq \Lo$. Denote $$I_\alpha=\{x\in \Lo |\; \forall{\beta \ge \alpha} \;x(\beta)=0 \}.$$
	
	Notice, that the sequence $\{I_\alpha\}_{\alpha<\omega_1}$ constitutes a filtration of $\Lo$. We define a filtration $\{L_\alpha\}_{\alpha<\omega_1}$ of $L$ by induction.
	\begin{itemize}
		\item $L_0 = I_0 \cap L$,
		\item Assume we have $L_\gamma$ constructed. Using Lemma 5, we can find $\tilde{L_\gamma} \subseteq \Lo$, with $\idim{\tilde{L_\gamma}}<\omega_1$, which patches $L_\gamma$. Define $L_{\gamma+1}=\tilde{L_\gamma} \cup (L\cap I_{\gamma})$. By Proposition 4, $\idim{L_{\gamma+1}}<\omega_1$.
		\item When $\gamma <\omega_1$ is limit, we set 
		$L_\gamma=\displaystyle{\bigcup_{\alpha<\gamma}{L_\alpha}}.$
	\end{itemize}
 Suppose now that $L$ is countably saturated, and has a filtration $\{L_\alpha\}_{\alpha<\omega_1}$. Let $X$ be any countably saturated linear order. Using Lemma On Bounded Injectivity, we can easily build an $\subseteq$-increasing sequence of increasing mappings $i_\alpha:L_\alpha \hookrightarrow X$. Its sum will be an embedding of $L$ into $X$.

\end{proof}

\begin{proof}[Proof of Theorem 12]
Consider two prime countably saturated linear orders. They both admit filtrations, so, using the Lemma On Bounded Injectivity, we can use back-and-forth argument, to inductively build an isomorphim between them.

\end{proof}

We can characterize the prime countably saturated linear order using the dimension.

\begin{prop}
	Let $(L,\le)$ be a countably saturated linear order. The following are equivalent:
	\begin{itemize}
		\item $L$ is prime countably saturated;
		\item $L$ is an increasing sum $\bigcup_{\alpha<\omega_1}{L_\alpha}$, where $\operatorname{2-dim}{L_\alpha}<\omega_1$, for each $\alpha<\omega_1$;
		\item  $L$ is an increasing sum $\bigcup_{\alpha<\omega_1}{L_\alpha}$, where $\idim{L_\alpha}< \omega_1$, for each $\alpha<\omega_1$.
	\end{itemize}
\end{prop}
\begin{proof}
	$1 \Rightarrow 2$. If $L$ is prime countably saturated, then \\$L \simeq \Lo = \bigcup_{\alpha<\omega_1}{[-1,1]^\alpha}\hookrightarrow \bigcup_{\alpha<\omega_1}{2^{\omega\cdot \alpha}}$.\\
	$2 \Rightarrow 3$. Clear.\\
	$3 \Rightarrow 1$. If $X$ is any countably saturated linear order, we construct an embedding $L \hookrightarrow X$ by induction, using the Lemma on Bounded Injectivity.
\end{proof}

This should by compared with the \emph{Characterization Theorem} of Harzheim \cite{harz}.

\begin{prop}[Harzheim, \cite{harz}]
	Let $(L,\le)$ be a countably saturated linear order. The following are equivalent:
	\begin{itemize}
		\item $L$ is prime countably saturated;
		\item $L$ is an increasing sum $\bigcup_{\alpha<\omega_1}{L_\alpha}$, where 
		$L_\alpha$ doesn't contain a copy of $\omega_1$ or $\omega_1^*$, for each $\alpha<\omega_1$.
	\end{itemize}
\end{prop}

There is one more propery of $\Lo$, which follows from the Lemma on Bounded Injectivity. 

\begin{prop}
	Any isomorphism between subsets of $\Lo$ of countable $\idim$ extends to an automorphism of $\Lo$. In particular, any automorphism between countable suborders extends to an automorphism of $\Lo$.
\end{prop}
\begin{proof}
	Let $\phi:X \hookrightarrow Y$ be in isomorphism, where $X,Y \subseteq \Lo$, and $\idim(X)<\omega_1$, $\idim(Y)<\omega_1$. We can define two filtrations $(X_\alpha)_{\alpha<\omega_1}$, $(Y_\alpha)_{\alpha<\omega_1}$, such that $X=X_0$, $Y=Y_0$. Then we successively extend the automorphism, using the Lemma on Bounded Injectivity.	
\end{proof}

As a matter of fact, Lemma on Bounded Injectivity basically proves that $\Lo$ is a Fra\"iss\'e limit of length $\omega_1$, as defined in \cite{fraisse}, of the class of linear orders of countable $\idim$.

\section{$(\omega_1,\mathfrak{c})$-saturated graphs}

 We investigate existence and uniqueness of graphs with certain homogeneity properties. All graphs denoted by $G$ or $H$ are undirected, those ones denoted by $\overleftarrow{G}$ or $\overleftarrow{H}$ are directed, and if $\overleftarrow{G}$ is a directed graph, then $G$ is the corresponding undirected graph, i.e. $\{a,b\}\in E(G)$ if and only if $(a,b)\in E(\overleftarrow{G})$ or $(b,a)\in E(\overleftarrow{G})$. By $N(g)$, where $g$ is a vertex of some graph, we denote the set of all vertices connected to $g$ (in whichever direction, in case the graph is directed). By $N_z(g)$ we denote the set of all vertices connected with $g$ by an arrow starting in $g$.

\par
\begin{defin}
	
	Let $\lambda$ and $\kappa$ be any cardinal numbers. A graph $G$ is $(\lambda,\kappa)$-saturated if for any $A,B\subset G$, $|A|<\lambda$, $|B|<\kappa$, $A\cap B=\emptyset$, there exists a vertex $v\in G$ such that $A\subseteq N(v)$, $ B\cap N(v)=\emptyset$. 
	
\end{defin}

It is long and well-known, that all countable $(\omega,\omega)$-saturated graphs are isomorphic. Under suitable cardinal arithmetic this result can be generalized to higher cardinalities. In particular if CH holds, then all $(\omega_1,\omega_1)$-saturated graphs of cardinality $\mathfrak{c}$ are isomorphic.

\begin{defin}
	We introduce some classes of graphs, useful for our purposes.
	\begin{itemize}
		
		\item A graph $G$ is of the first type, if there exists a bijective enumeration $G=\{g_\alpha|\, \alpha<\mathfrak{c}\}$, such that:
		\begin{enumerate}
			\item $|N(g_\alpha)\cap \{g_\beta|\beta<\alpha\}|\le \omega,$
			for all $\alpha<\mathfrak{c};$
			\item for each countable $A\subseteq G$, there are continuum many indices $\alpha$, such that \\
			$N(g_\alpha)\cap \{g_\beta|\beta<\alpha\}=A$.
		\end{enumerate}
		\item A graph $\overleftarrow{G}$ is strictly saturated, if
		\begin{enumerate}
			\item $|N_z(v)|\le \omega$, for each $v\in \overleftarrow{G}$;
			\item for each countable $A\subseteq
			\overleftarrow{G}$, there are $|\overleftarrow{G}|$-many $g \in \overleftarrow{G}$, such that $N_z(g)=A$.
		\end{enumerate}
	\end{itemize}
\end{defin}

As in the case of linear orders, a graph from some class of graphs $\mathcal{C}$ is prime, if it embeds into every graph from $\mathcal{C}$. We aim to prove the existence and uniqueness of the prime $(\omega_1,\mathfrak{c})$-saturated graph.\\

Recall, that the \emph{colouring number}, denoted $\col(G)$ of a graph $G$ is the least cardinal $\kappa$, such that there exists a well-ordering $(G,\preccurlyeq)$, with the property $$|\{ h\preccurlyeq g|\, \{g,h\}\in E(G) \}|<\kappa,$$ for all $g \in G$. In follows from the definition, that any graph of the first type has colouring number $\omega_1$.

\begin{prop}
	If a graph $G$ is of the first type, then the edges of $G$ can be directed in such a way, that the resulting directed graph is strictly saturated, and doesn't have an infinite path.
\end{prop}

\begin{proof}
	Fix a well-ordering, which witnesses that $G$ is of the first type, and then direct every edge in a decreasing manner with respect to that ordering.
\end{proof}

 This is an important observation, and this is why.

\begin{lem}
	If $\overleftarrow{G}$ and $\overleftarrow{H}$ are strictly saturated graphs of cardinality $\mathfrak{c}$ without infinite paths, then they are isomorphic.
\end{lem}
The proof will slightly resemble the proof of Mostowski Collapse Lemma.
\begin{proof}
	Let $\{g_\alpha|\alpha<\mathfrak{c}\}$ and $\{h_\alpha|\alpha<\mathfrak{c}\}$ be any bijective enumerations of $\overleftarrow{G}$ and $\overleftarrow{H}$ respectively.
	By induction we construct an isomorphism $\phi:\overleftarrow{G}\rightarrow \overleftarrow{H}$. Assume, that we are in step $\eta<\mathfrak{c}$.
	\begin{itemize}
		\item If $\eta$ is even, then choose the least $\alpha$ such that $\phi(g_\alpha)$ has not yet been defined, but $\phi\restriction N_z(g_{\alpha})$ is defined (such element always exists, because there's no infinite path). By strict saturation we can choose $h\in\overleftarrow{H}$, for which $\phi^{-1}(h)$ is not yet defined, and $N_z(h)=\phi[N_z(g_{\alpha})]$. Set $\phi(g_{\alpha})=h$.
		\item If $\eta$ is odd, then choose the least $\alpha$, such that $\phi^{-1}(h_\alpha)$ is has not been defined, but $\phi^{-1}\restriction N_z(h_{\alpha})$ is defined. Again, by strict saturation we can choose $g\in \overleftarrow{G}$, which hasn't yet been used, and $N_z(g)=\phi^{-1}[N_z(h_{\alpha})]$. Set $\phi(g)=h_\alpha$.	
	\end{itemize}
	A standard verification shows that $\phi$ is indeed a bijection between $\overleftarrow{G}$ anf $\overleftarrow{H}$. Indeed, suppose that some vertex $x_0 \in \overleftarrow{G}$ somehow evaded being assigned the image $\phi(x_0)$. Then, there must be some $x_1\in \overleftarrow{G}$, such that $x_0\rightarrow x_1$, and which shared this fate. If not, then $\phi \restriction N_z(x_0)$ would have been defined, and since $\operatorname{cf}{\mathfrak{c}}> \omega =|N_z(x_0)|$, it would have been defined at a certain step. From that moment on, $x_0$
would be ready to be chosen next. Since there are less than $\mathfrak{c}$
elements with priority below $x_0$, and $\mathfrak{c}$ future steps, its turn would
eventually arrive. But then we apply the same argument to $x_1$, getting $x_1 \rightarrow x_2$, and so on, producing an infinite path. Because of symmetry of the construction, the same argument shows, that $\phi$ is onto. \\
	
	 We will check, that it is a homomorphism. Fix an arrow $g_\alpha \rightarrow g_{\beta}$ in $\overleftarrow{G}$. The construction assures, that $\phi(g_\beta)$ is defined before $\phi(g_\alpha)$.\\
	\begin{enumerate}
		\item Suppose, that we are defining $\phi(g_\alpha)$ in an even step. Then 
		$\phi(g_\alpha)\rightarrow \phi(g_\beta)$ is in $\overleftarrow{H}$, by the definition of $\phi(g_\alpha)$.
		\item Suppose, that we are defining $\phi(g_\alpha)$ in an odd step. Then, in this step  we are setting $\phi^{-1}(h_\gamma)=g_\alpha$ for some $\gamma<\mathfrak{c}$. Then $g_\beta\in N_z(g_\alpha)=\phi^{-1}[N_z(h_\gamma)]$. Therefore, $\phi(g_\beta)\in N_z(h_\gamma)$, and $h_\gamma=\phi(g_\alpha)$. So $\phi(g_\alpha)\rightarrow \phi(g_\beta)$ is in $\overleftarrow{H}$.

	\end{enumerate}
	Because the construction was symmetric, exactly the same argument shows that if $g_\alpha$ and $g_\beta$ were not connected, $\phi(g_\alpha)$ and $\phi(g_\beta)$ are not connected either. This finishes the proof.
\end{proof}

We can now give the internal characterization of prime strictly saturated graphs.

\begin{thm} A strictly saturated graph of cardinality $\mathfrak{c}$ is prime if and only if it has no infinite path.
\end{thm}
\begin{proof}
	$\Rightarrow.$ The property of having no infinite path is hereditary, so if some strictly saturated graph has this property (and we will soon see, that it has), then so does every prime one. \\
	$\Leftarrow.$ We do the same construction as in the proof above, but only in one direction, so that we obtain an embedding instead of an isomorphism. The only thing which is not immediate, is why if $g_\alpha$ and $g_\beta$ are not connected, $\phi(g_\alpha)$ and $\phi(g_\beta)$ are not connected either. But we are choosing $\phi(g_\alpha)$ in a way that $N_z(\phi(g_\alpha))=\phi[N_z(g_\alpha)]$. $g_\beta \notin N_z(g_\alpha)$, so $\phi(g_\beta) \notin N_z(\phi(g_\alpha)).$
\end{proof}

\begin{cor}
	All prime strictly saturated graphs are isomorphic.
\end{cor}
From Proposition 9 follows
\begin{cor}
	All graphs of the first type are isomorphic.
\end{cor}
\begin{thm}
	There exists a prime strictly saturated graph without an infinite path. Moreover, a graph $G$ is of the first type if and only if it can be assigned a direction of edges, such that the resulting directed graph 
	is prime strictly saturated.
\end{thm}

\begin{proof}
	All graphs of the first type are isomorphic, and also all prime strictly saturated graphs are isomorphic. We therefore need only to find some strictly saturated graph $\overleftarrow{G}$ with no infinite path, such that $G$ is of the first type. We apply a bookkeeping argument, similar to the one used in the proof of existence of Avil\'es-Brech Boolean algebra in \cite{ab}. Let $\{\Phi_\alpha|\alpha<\mathfrak{c}\}$ be a partition of $\mathfrak{c}$ into sets of cardinality $\mathfrak{c}$, and such that $\min{\Phi_\alpha}\ge \alpha$. 
	\begin{itemize}
		\item $G_0=\{0\}$, and for $\gamma \in \Phi_0$, we set $S_\gamma=\{0\}$.
		\item If $\alpha$ is a limit ordinal, let $G_\alpha =\bigcup\{G_\beta|\beta < \alpha\}$, and $\{S_\gamma|\gamma \in \Phi_\alpha\}$ be an enumeration of all countable subsets of $G_\alpha$, in which every set appears $\mathfrak{c}$ many times.
		
		\item Assume that $G_\alpha$ and $S_\gamma$, for $\gamma \in \bigcup{\{\Phi_\beta|\beta\le\alpha\}}$, have been defined. We set $G_{\alpha+1}=G_{\alpha}\cup\{\alpha\}$, and $\{\alpha,g\}\in G_{\alpha+1}$ iff $g\in S_\alpha$, where $g\in G_\alpha$. Finally, we use $\Phi_{\alpha+1}$ to list all countable subsets of $G_{\alpha+1}$, and each of them $\mathfrak{c}$-many times.
	\end{itemize}
	A natural ordering of $G$ given by the well-ordering of the ordinal $\mathfrak{c}$ witnesses, that $G$ is of the first type, and therefore can be directed into $\overleftarrow{G}$, strictly saturated graph without infinite path, by directing every edge downwards with respect to the well-ordering of $\mathfrak{c}$.
\end{proof}

Note, that this argument actually shows, that for any cardinal $\kappa$ such that $\kappa^\omega=\kappa$, there exists a strictly saturated graph of cardinality $\kappa$ without an infinite path, and all such graphs are isomorphic.

We obtain an internal characterization of prime $(\omega_1,\mathfrak{c})$-saturated graphs.

\begin{thm}
	An $(\omega_1,\mathfrak{c})$-saturated graph is prime $(\omega_1,\mathfrak{c})$-saturated if and only if it is of the first type.
\end{thm}

\begin{proof}
	$\Leftarrow .$ Let $G$ be the graph of the first type. For simplifying notation, let us assume that $V(G)=\mathfrak{c}$, and its natural well-ordering witnesses the first type of $G$. \\
	First, we check that if $G$ is of the first type, it is actually $(\omega_1,\mathfrak{c})$-saturated. Let $A\subseteq G$ be countable and $B\subseteq G$ be of cardinality less than $\mathfrak{c}$, disjoint with $A$. For all $b\in B$, let $C_b=N(b)\cap b$, and $S=\bigcup\{C_b|\, b\in B\}$. $|S|<\mathfrak{c}$, so there exists $g \in G\setminus S$, such that $N(g)\cap g=A$. Clearly $N(g)\cap B=\emptyset$. 
	\\ Verification that $G$ embeds into any $(\omega_1,\mathfrak{c})$-saturated graph is proved by a standard transfinite induction argument.
	 \\ $\Rightarrow .$ Consider a prime $(\omega_1,\mathfrak{c})$-saturated graph. Without loss of generality we can assume that it is an induced subgraph of $G$. Our goal for the moment is to prove, that we can direct its edges in such a way, that we obtain the prime strictly saturated graph. The property of being of the first type is not hereditary, but $(\omega_1,\mathfrak{c})$-saturated subgraph of $G$ will be satisfying some property close enough to being of the first type.
	
	\begin{lem}
		Let $G$ be $(\omega_1,\mathfrak{c})$-saturated, such that $V(G)=\mathfrak{c}$, and for any $\alpha \in G$, $|N(\alpha)\cap \alpha| \le \omega$. Then $G$ can be directed into the prime strictly saturated graph.
	\end{lem}
	\begin{proof}
		At the beginning, we direct all edges of $G$ downwards with respect to  the well-ordering of $\mathfrak{c}$. We will inductively redirect some edges, ensuring that every countable subset will appear as $N_z(\alpha)$, but no infinite path will be added.
		\\ So assume that $\overleftarrow{G}$ is our graph $G$ with edges directed downwards. Denote by $A(\beta)$ the set of all vertices accessible via a finite, decreasing path from $\beta$ (note that adjective "decreasing" may sound superfluous, since all edges are directed downwards, but by the inductive procedure we will perform, some arrows will be reversed). It is clearly countable. Let $\{C_\alpha\}_{\alpha<\mathfrak{c}}$ be an enumeration of all countable subsets of $\mathfrak{c}$, such that each set appears $\mathfrak{c}$ many times. By induction we will choose the sequence $(x_\alpha)_{\alpha<\mathfrak{c}}$, satisfying what follows.
		\begin{enumerate}
			\item $x_\alpha \notin \{x_\beta|\,\beta<\alpha \}$;
			\item $C_\alpha \subseteq N(x_\alpha)\cap x_\alpha$;
			\item $N(x_\alpha)\cap \displaystyle{
				\bigcup_{\eta\le\alpha}
				\bigcup_{\beta \in C_\eta}{(A(\beta) \setminus C_\alpha)}
			}=\emptyset$;
			\item $(N(x_\alpha)\setminus C_\alpha)  \cap \bigcup \{N(x_\beta)\cap x_\beta|\,\beta<\alpha \}=\emptyset$;
			\item $x_\alpha \notin \displaystyle{\bigcup_{\beta<\alpha}{N(x_\beta)\cap x_\beta}}$;
			\item $N(x_\alpha)\cap(\{x_\beta|\,\beta<\alpha\}\setminus C_\alpha)=\emptyset$.
		\end{enumerate}
		At each step the choice can be made because of $(\omega_1,\mathfrak{c})$-saturation.  Once $x_\alpha$ has been chosen, we reverse all arrows $c\leftarrow x_\alpha$ for $c\in N(x_\alpha)\cap x_\alpha\setminus C_\alpha$. It follows from 2. that every countable subset of $\mathfrak{c}$ will appear as $N_z(\alpha)$ for $\mathfrak{c}$ many $\alpha$. The only things which are left to check, is that we won't produce an infinite path, and that we will preserve the property $|N_z(\alpha)|\le \omega$ for $\alpha<\mathfrak{c}$. Since in the $\alpha$-th step we reverse only arrows from $N(x_\alpha)\cap x_\alpha \setminus C_\alpha$, for the second part, it is sufficient to prove
		
		\begin{claim}
			The sets $N(x_\alpha)\cap x_\alpha \setminus C_\alpha$ are pairwise disjoint for $\alpha<\mathfrak{c}$.
		\end{claim}
		\begin{proof}
			Follows directly from 4.
		\end{proof}
		
		\begin{claim}
			Each arrow was reversed no more than once.
		\end{claim}
		\begin{proof}
			If we are in step $\alpha$ and the arrow $c\leftarrow x_\alpha$ is about to be reversed for the second time, then clearly $c=x_\beta$ for some $\beta<\alpha$. But this contradicts 6.
		\end{proof}
		\begin{claim}
			There is no situation where $\alpha<\beta<\gamma <\mathfrak{c}$, and $\alpha\rightarrow \beta \rightarrow \gamma$.
		\end{claim}
		\begin{proof}
			Suppose otherwise. Then $\beta=x_\delta$ and $\gamma=x_\eta$, for some $\delta,\eta <\mathfrak{c}$. If $\delta <\eta$, then 
			$\beta \in (\{x_\rho|\,\rho<\eta\}\setminus C_\eta)\cap N(x_\eta)$, contradicting 6. If $\eta<\delta$, then $\beta =x_\delta \in N(x_\eta)\cap x_\eta$, which contradicts 5.
		\end{proof}
		\begin{claim}
			If $\alpha<\beta$, $\gamma<\beta$, $\gamma<\delta$, and we have arrows $\alpha \rightarrow\beta $, $\gamma \rightarrow \delta$, and a finite, decreasing path from $\beta$ to $\gamma$. Then there exists $\eta,\theta < \mathfrak{c}$, such that $x_\eta=\beta$, $x_\theta=\delta$, and $\theta<\eta$.
		\end{claim}
		\begin{proof}
			Existence is clear, since otherwise the corresponding arrows wouldn't be reversed. Suppose, that $\theta\ge\eta$. Then 
			$\gamma \in N(x_\theta)\cap\displaystyle{\bigcup_{\epsilon \in C_\eta}{A(\epsilon)\setminus C_\theta}}$. But this contradicts 3. 
		\end{proof}
		\begin{claim}
			The graph $G$ has no infinite path.
		\end{claim}
		\begin{proof}
			Follows from the previous three Claims, and the fact that $\mathfrak{c}$ is well-ordered.
		\end{proof}			
		This concludes the proof of the last Claim.	
	\end{proof}
	Theorem 15 follows at once from Lemma 8.
\end{proof}

Putting everything into one place, we obtain \\

\begin{thm}
	Let $G$ be an $(\omega_1,\mathfrak{c})$-saturated graph. The following are equivalent.
	
	\begin{enumerate}
		\item $G$ is prime $(\omega_1,\mathfrak{c})$-saturated.
		\item $\col(G)=\omega_1$.
		\item $G$ is of the first type.
		\item $G$ has a direction of edges such that $\overleftarrow{G}$ is prime strictly saturated.
		\item $G$ has a direction of edges such that $\overleftarrow{G}$
		is strictly saturated, without an infinite path.
	\end{enumerate}
	Moreover, any of these properties determines the graph $G$ uniquely up to isomorphism.
\end{thm}

The natural question that arises, is whether similar conclusion holds for $(\omega_1,\omega_1)$-saturated graphs.

\begin{thm}
	The prime $(\omega_1,\omega_1)$-saturated graph exists only if $\mathfrak{c}=\omega_1$.
\end{thm}

\begin{proof}
	Let $\bar{\cdot}$ denote the complement of a graph. Let $G$ be the prime $(\omega_1,\mathfrak{c})$-saturated graph. Of course both $G$ and $\bar{G}$ are $(\omega_1,\omega_1)$-saturated, so it is enough to prove, that no graph of cardinality $\omega_2$ appears as an induced subgraph of both of them. Notice, that if $S$ were such graph, then $\col(S)\le \omega_1$ and $\col(\bar{S})\le \omega_1$. So it is sufficient to show, that this cannot be the case.
	\\Suppose otherwise, that $S$ is a graph of cardinality $\omega_2$, with $\col(S)\le \omega_1$ and $\col(\bar{S})\le \omega_1$. Fix well-orderings $S=\{s_\alpha|\, \alpha<\omega_2\}=\{t_\alpha|\, \alpha<\omega_2\}$ witnessing the colouring number of $S$ and $\bar{S}$ respectively, and denote by $\sigma:\omega_2 \rightarrow \omega_2$ the bijection given by $\forall_{\alpha<\omega_2} s_\alpha=t_{\sigma(\alpha)}$. Finally, fix an ordinal $\gamma>\omega_1$, for which also $\sigma(\gamma)>\sigma[\omega_1]$. 
	\\Let $A=\{\alpha < \omega_1|\, \{s_\alpha,s_\gamma\}\in E(S)\}$, and $B=\{\alpha < \omega_1|\, \{s_\alpha,s_\gamma\}\notin E(S)\}$. Clearly $A$ is countable. Also $B=\{\alpha<\omega_1|\{t_{\sigma(\alpha)},t_{\sigma(\gamma)}\}\notin E(S)\}= \sigma^{-1}[\{\alpha\in \sigma[\omega_1]|\{t_\alpha,t_{\sigma(\gamma)}\}\notin E(S)\}]$, which is countable. But $A\cup B = \omega_1$, which is a contradiction.
\end{proof}

\subsection{The concrete representation}
	
	Consider the set of all hereditarily countable sets $H(\omega_1)$, i.e. the countable sets, of which every element is countable, and every element of element is countable, and so forth. This structure, and its higher analogs $H(\kappa)$, appear in the axiomatic set theory, because $(H(\kappa),\in)$ is a natural model of ZFC without the powerset axiom, for any uncountable, regular $\kappa$. For more elaborated introduction, we refer the reader to \cite{kunen}. It was noticed by Arturo Martinez-Celis, that the prime $(\omega_1,\mathfrak{c})$-saturated graph we have been studying is isomorphic to the graph where set of vertices is $H(\omega_1)$, and an edge is drawn between $x$ and $y$ if and only if $x\in y$ or $y\in x$.
	
\begin{thm}
	Let $(H(\omega_1),\in \cup \ni)$ be the graph defined above. It is prime $(\omega_1,\mathfrak{c})$-saturated.
\end{thm}
\begin{proof}
	First, let us check, that it is $(\omega_1,\mathfrak{c})$-saturated. Let $A,B \subseteq H(\omega_1)$, $|A|\le \omega$, $|B|<\mathfrak{c}$. We aim to find a countable set $A'\supseteq A$, disjoint with $B$, and not belonging to $\bigcup{B}$.	But $|\bigcup{B}|<\mathfrak{c}$, and there are $\mathfrak{c}$-many different hereditarily countable sets having $A$ as a subset, and disjoint with $B$. \\
	We now prove, that it is prime. It is sufficient to prove that for any strictly saturated graph $G$, we have an embedding $\phi:(H(\omega_1),\in) \hookrightarrow G$. We define $\phi$ by the induction, with respect to the $\in$ relation. Suppose that $x \in H(\omega_1)$, and for every $y\in x$, $\phi(y)$ is defined. Set $\phi(x)=g$, where $N_z(g)=\phi[x]$. It is routine to check, that this is an embedding.
\end{proof}

\begin{cor}
	$|H(\omega_1)|=\mathfrak{c}.$
\end{cor}	

 Note, that it is not even \emph{a priori} clear, that $H(\omega_1)$ is a set, not a proper class.
	
\section{Boolean algebras with the strong countable separation property}

In this section letters $A,B,C$ always denote Boolean algebras.

\begin{defin}
	An infinite Boolean algebra $A$ has the strong countable separation property if the following assertion holds:
	
	Suppose $F,G \subseteq A$ are countable sets with the property that for all nonempty, finite subsets $f \subseteq F$, $g\subseteq G$, $\bigvee{f}<\bigwedge{g}$. Then there exist $a\in A$, such that for all $x\in F$, and $y\in G$, $x<a<y$.
\end{defin}

This is equivalent to being injective with respect to countable subalgebras.

\begin{prop}[5.29 in \cite{bool}] Assume that $A$ has the strong countable separation property, $f:B\rightarrow A$ is an embedding, and $B\subseteq C$ is countable. Then $f$ extends to an embedding $\tilde{f}:C\rightarrow A$.
\end{prop}

For exposition of general theory of Boolean algebras with the strong countable separation property, we refer the reader to \cite{bool}. In case of Boolean algebras, we start from the classical characterization of
$\mathcal{P}(\omega) \slash \text{Fin}$ by Parovi\v{c}enko.

\begin{thm}[Parovi\v{c}enko] Under CH, $\mathcal{P}(\omega) \slash \text{Fin}$ is the unique Boolean algebra of cardinality $\mathfrak{c}$, with the strong countable separation property.
\end{thm}

In fact, this statement is equivalent to CH, as was proved by van Douwen and van Mill \cite{van}. In \cite{ab} Avil\'es and Brech consider Boolean algebras which realize extensions by so-called posex' (push-out separable extensions).

\begin{defin}[Avil\'es-Brech, \cite{ab}]
	An embedding $f:A\rightarrow B$ is a posex, if the following holds:
	\begin{itemize}
		\item for all $b\in B\setminus f[A]$, the ideal $\{a\in A|\, f(a) <b\}$ is countably generated,
		\item $B$ is countably generated over $f[A]$ (i.e. there exist a countable set $S \subseteq B$, such that $B$ is generated by $f[A]\cup S$).
	\end{itemize}
\end{defin}

\begin{defin}[Avil\'es-Brech, \cite{ab}]
	A Boolean algebra $A$ is tightly $\sigma$-filtered if there exists a sequence of Boolean algebras $(A_\eta)_{\eta\le \zeta}$, such that
	\begin{itemize}
		\item $A_0=\{0,1\}$,
		\item $A_\eta \subseteq A_{\eta+1}$ is a posex, for $\eta<\zeta$,
		\item $A_\eta =\bigcup_{\alpha<\eta}{A_\alpha}$, for limit $\eta<\zeta$,
		\item $A_\zeta=A$.		
	\end{itemize}
	We call the sequence $(A_\eta)_{\eta<\zeta}$ a filtration of $A$.
\end{defin}

\begin{thm}[Avil\'es-Brech, \cite{ab}]
	Assume, that $\mathfrak{c}$ is regular. Then there exists a unique Boolean algebra $A$, such that:
	\begin{itemize}
		\item $|A|=\mathfrak{c}$,
		\item $A$ is tightly $\sigma$-filtered,
		\item For any embedding $f:B\rightarrow A$, where $|B|<\mathfrak{c}$, and any posex $C\supseteq B$,
		there exists an embedding $\tilde{f}:C\rightarrow A$, extending $f$.
	\end{itemize}
	
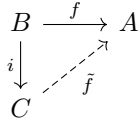
\begin{tikzcd}
B \arrow{r}{f} \arrow{d}[swap]{i} & A\\
C \arrow[dashed]{ur}[swap]{\tilde{f}}
\end{tikzcd}

\end{thm}

 We will call an algebra with these properties an Avil\'es-Brech algebra. Since any extension between countable Boolean algebras is posex, the algebra defined above clearly has the strong countable separation property. Moreover, it is prime with respect to this property.

\begin{thm} The Avil\'es-Brech algebra defined above, embeds into any Boolean algebra having the strong countable separation property.
\end{thm}

\begin{proof}
	Let $B$ be any Boolean algebra with the strong countable separation property. Let $A$ be an Avil\'es-Brech algebra, with filtration $(A_\alpha)_{\alpha<\mathfrak{c}}$. We construct an increasing sequence of embeddings $f_\beta : A_\beta \rightarrow B$, by induction on $\alpha$, using the Corollary 5.8 from \cite{bool}, quoted below.
	\begin{lem}
	Let $A$ be any Boolean algebra and $A(x)$ its extension generated
	by $A\cup \{x\}$. Let $f:A\rightarrow B$ be an embedding, and let $y \in B$. Then $f$ extends to an embedding $\tilde{f}:A(x)\rightarrow B$, with $\tilde{f}(x)=y$ if and only if for all $a,a' \in A$, $a \le x \le a'$ implies $f(a)\le y \le f(a')$.
	\end{lem}
	Suppose that $f_\beta :A_\beta \rightarrow B$ is defined for all $\beta<\alpha <\mathfrak{c}$. If $\alpha$ is limit, we set $f_\alpha=\bigcup\{f_\beta:\, \beta<\alpha\}$, so suppose that $\alpha = \beta+1$. Let $S=\{s_0, s_1, \ldots\}$ be a set such that $A_\alpha = \langle  A_\beta \cup S \rangle$. We successively extend  $f_\beta$ to each of the generators using the Lemma. More precisely, let $f_\beta^0=f_\beta$, and assume that $f_\beta^i : \langle A_\beta \cup \{s_0,\ldots,s_{i-1}\} \rangle \rightarrow B$ is defined. Let $\{a_n|\, n<\omega\}$ be a set of generators of the ideal $ \{a \in \langle A_\beta \cup \{s_0,\ldots,s_{i-1}\}\rangle |\, f_\beta^i(a) < s_i \}$, and let
	$\{a'_n |\, n<\omega \}$ be a set of generators of the ideal $ \{a \in \langle A_\beta \cup \{s_0,\ldots,s_{i-1}\}\rangle |\, f_\beta^i(a) < -s_i \}$. Because of the strong countable separation property, there exists $b \in B$ such that for all $n<\omega$
	$$f_\beta^i(a_n)<b<-f_\beta^i(a'_n).$$
	We extend $f_\beta^i$ to $f_\beta^{i+1}$, so that $f_\beta^{i+1}(s_i)=b$. Finally, let $f_\alpha = \bigcup_{i<\omega}{f_\beta^i}$.

\end{proof}

 A natural question arises.
 
 \begin{prob}
 	Let $A$ be a Boolean algebra with the strong countable separation property, and such that whenever $B$ has the strong countable separation property, then $A$ embeds into $B$. Does it follow that $A$ is the Avil\'es-Brech algebra?
 \end{prob}


\begin{thebibliography}{HD}
	
	
	\baselineskip=15pt
	
	\begin{small}
		
			\bibitem[1]{ab} A. Avil\'es, C. Brech, \textit{A Boolean algebra and a Banach space obtained by push-out iteration}, Topology and its Applications 158 (2011), 1534-1550
			
			\bibitem[2]{cohen} A. Dow, K.P. Hart, \textit{Applications of another characterization of $\beta \mathbb{N} \setminus \mathbb{N}$}, Topology and its Applications 122 (2002), 105-133
			
			\bibitem[3]{giarlotta}A. Giarlotta, \textit{The representability number of a chain}, Topology and its Applications 150 (2005), 157-177
			
			\bibitem[4]{Gillman}L. Gillman, \textit{Some remarks on $\eta_\alpha$-sets}, Fund. Math. 43 (1956), 77-82
			
			\bibitem[5]{harz} E. Harzheim, \textit{Ordered sets}, Springer, 2005, 97-108
			
					
			\bibitem[6]{Hausdorff}F. Hausdorff, \textit{Gr\"undzuge einer Theorie geordneten Mengen}, Math. Ann. 65 (1908), 435-505
			
			\bibitem[7]{hod} W. Hodges, \textit{A shorter model theory}, Cambridge University Press, 1997
			
			\bibitem[8]{bool} S. Koppelberg, edited by D. Monk with R. Bonnet, \textit{Handbook of Boolean algebras, vol. 1}, North-Hollad, 1989
			
			
			\bibitem[9]{fraisse} W. Kubi\'s, \textit{Fraïss\'e sequences: category-theoretic approach to universal homogeneous structures}, Annals of Pure and Applied Logic 165(11), 2014 1755-1811
			
			
			\bibitem[10]{kunen} K. Kunen, \textit{Set Theory, Introduction to Independence Proofs}, Elsevier Science B.V., 1980
			
			\bibitem[11]{novak}V. Nov\'ak, \textit{On the lexicographic dimension of linearly ordered sets}, Fund. Math. 56 (1964), 8-20
			
			\bibitem[12]{novotny}M. Novotn\'y, \textit{On similarity of ordered continua of types $\tau$ and $\tau^2$}, Československá Akademie Věd. Časopis Pro Pěstování Matematiky, 78 (1953), 59-60
			
			
			
			\bibitem[13]{rosen}J. G. Rosenstein, \textit{Linear Orderings}, Academic Press Inc., 1982
		
			\bibitem[14]{Urysohn}P. Urysohn, \textit{Un th\'eor\'eme sur la puissance des ensembles ordonn\'es}, Fund. Math. 5 (1923), 14-19
		
		
			\bibitem[15]{van} E.K. van Douwen, J. van Mill,, \textit{Parovi\v{c}enko's Characterization of $\beta \omega - \omega$ Implies CH}, Proc. Am. Math. Soc, Vol. 72, No. 3 (1978), 539-541
		
		
		
		
		
	
		
		
		

		
		
		

		
	\end{small}
\end{thebibliography}
\end{document}